\documentclass[11pt]{amsart}
\usepackage{amsmath}
\usepackage{amssymb}
\usepackage{amscd}
\usepackage{hyperref}

\let\=\partial

\newtheorem{theorem}{Theorem}[section]
\newtheorem{lemma}[theorem]{Lemma}

\newtheorem{claim}[theorem]{Claim}

\theoremstyle{definition}
\newtheorem{definition}[theorem]{Definition}

\newtheorem{definitions and remarks}[theorem]{Definitions and Remarks}

\theoremstyle{remark}
\newtheorem{remark}[theorem]{Remark}

\numberwithin{equation}{section}

\newcommand{\ord}{\mathrm{ord}}

\newcommand{\al}{{\alpha}}

\newcommand{\ep}{{\epsilon}}

\newcommand{\g}{{\gamma}}
\newcommand{\Ga}{{\Gamma}}

\newcommand{\la}{{\lambda}}

\newcommand{\vp}{{\varphi}}

\newcommand{\IN}{{\mathbb N}}
\newcommand{\IQ}{{\mathbb Q}}

\newcommand{\IR}{{\mathbb R}}

\newcommand{\cC}{{\mathcal C}}

\newcommand{\cI}{{\mathcal I}}
\newcommand{\cJ}{{\mathcal J}}

\newcommand{\cM}{{\mathcal M}}

\newcommand{\cQ}{{\mathcal Q}}
\newcommand{\cR}{{\mathcal R}}

\begin{document}
\title[Arc-quasianalytic functions]{Arc-quasianalytic functions}

\author[E.~Bierstone]{Edward Bierstone}
\author[P.D.~Milman]{Pierre D. Milman}
\author[G.~Valette]{Guillaume Valette}

\address{University of Toronto, Department of Mathematics, 40 St. George Street,
Toronto, ON, Canada M5S 2E4}
\email[E.~Bierstone]{bierston@math.toronto.edu}
\email[P.D.~Milman]{milman@math.toronto.edu}
\address{Instytut Matematyczny PAN, ul. \'Sw. Tomasza 30, 31-027 Krak\'ow,
Poland}
\email[G.~Valette]{gvalette@impan.pl}
\thanks{Research supported in part by NSERC grants MRS342058, OGP0009070, and
OGP0008949.}



\begin{abstract}
We work with quasianalytic classes of functions. Consider a real-valued function $y = f(x)$
on an open subset $U$ of $\IR^n$, which satisfies a quasianalytic equation $G(x,y) = 0$.
We prove that $f$ is \emph{arc-quasianalytic} (i.e., its restriction to 
every quasianalytic arc is quasianalytic) if and only if
$f$ becomes quasianalytic after (a locally finite covering of $U$ by)
finite sequences of local blowings-up. This generalizes a theorem of the first two authors
on arc-analytic functions.
\end{abstract}

\date{\today}
\maketitle
\setcounter{section}{0}

\section{Introduction}
\emph{Arc-analytic} functions are functions that are analytic along every analytic arc. Arc-analytic
functions were 
introduced by K. Kurdyka to study the geometry of \emph{arc-symmetric} sets \cite{K}.
The first two authors proved that a real-valued function $f$ on a smooth real-analytic variety $U$ 
is arc-analytic and has subanalytic graph if and only if $f$ becomes analytic after a locally
finite covering of $U$ by finite sequences of local blowings-up \cite[Thm.\,1.4]{BMarc}. The latter 
has become a basic tool
in real-analytic geometry (see, for example, \cite{KPar}, \cite{KPau} and other articles 
referenced therein). This paper deals with a
natural generalization of arc-analytic functions to \emph{quasianalytic classes}, and establishes
the analogue of \cite[Thm.\,1.4]{BMarc} for such \emph{arc-quasianalytic} functions; see
Theorem \ref{thm:main} below.

A \emph{quasianalytic class} associates, to every open
subset $U \subset \IR^n$, a subring $\cQ(U)$ of $\cC^\infty(U)$
which satisfies the basic properties of $\cC^\infty$ functions 
together with the property that the Taylor series homomorphism at any point of $U$ is injective.
See Section \ref{sec:quasian} for a precise definition. Examples are quasianalytic 
Denjoy-Carleman classes that are closed under differentiation (see \cite{BMsel}), and the class of $\cC^\infty$
functions which are definable in a given polynomially bounded $o$-minimal structure \cite{RSW}. The
former play an important part in analysis; in particular, in the study of certain partial differential
equations. 

Quasianalytic classes are in general much bigger than the class of real-analytic functions,
but nevertheless enjoy many properties of analytic functions. In \cite{BMinv}, \cite{BMsel}, the first two
authors proved resolution of singularities for finitely-generated ideals in quasianalytic classes.
The latter can be used to show that \emph{quasianalytic sets} (i.e., sets
defined by finitely many functions in a quasianalytic class) have geometric properties similar
to those of analytic sets. In particular, the corresponding Zariski topology is Noetherian \cite{BMsel}, though
it seems unknown (and doubtful) whether rings of germs of quasianalytic functions are in general Noetherian.
We can show nevertheless that resolution of singularities holds even for quasianalytic ideals
that are not necessarily finitely generated; see Theoreom \ref{thm:res}, which is an important
tool in the proof of our main theorem following.

Let $\cQ$  denote a given quasianalytic class (Section \ref{sec:quasian}).

\begin{definition} 
Let $W$ be an open subset of $\IR^n$. A function $f :W \to \IR$ is called \emph{arc-quasianalytic} if 
$f\circ \g \in \cQ((-\ep,\ep ))$, for every \emph{quasianalytic arc} $\g: (-\ep,\ep) \to W$ (the
latter means that
$\g = (\gamma_1,\ldots,\gamma_n)$, where each $\g_i \in \cQ((-\ep,\ep))$).
\end{definition}

We will say that a family of quasianalytic mappings $\{\pi_j:U_j \to U\}$ is a \emph{locally finite covering} of $U$ if 
(1) the images $\pi_j(U_j)$ are subordinate to a locally finite covering
of $U$ by open subsets; (2) if $K$ is a compact subset of
$U$, then there are compact subsets $K_j$ of $U_j$, for each $j$, such that
$K=\cup \pi_j(K_j)$ (the union is finite, by (1)). 

A \emph{modification} will denote a finite composite of \emph{admissible local blowings-up}.
(A \emph{local blowing-up} of $U$ is a blowing-up over an open subset of $U$. A (local)
blowing-up is \emph{admissible} if its centre is smooth and normal crossings with the exceptional
divisor.)

\begin{theorem}\label{thm:main}
Let $f:U \to \IR$ denote a function on a connected open set $U\subset \IR^n$.
Assume there is a nonzero quasianalytic function $G: U\times \IR \to \IR$ such that 
$G(x,f(x))\equiv 0$.  Then $f$ is arc-quasianalytic  if and only if
there exists a locally finite covering $\{\pi_j: U_j \to U\}$ such that, for each $j$,
\begin{enumerate}
\item $\pi_j$ is a modification;
\item $f\circ \pi_j$ is quasianalytic.
\end{enumerate}
\end{theorem}

The proof of \cite[Thm.\,1.4]{BMarc} relies on a generalization of
Hensel's lemma to several variables and thus makes use of the Weierstrass
preparation theorem. The latter does not hold in quasianalytic classes in general
\cite{Child}. We were therefore forced to imagine a rather different proof of
Theorem \ref{thm:main}, involving
a more technical iterative argument.

K. Nowak has used Theorem \ref{thm:main} to prove an interesting result on hyperbolic 
polynomials with quasianalytic coefficients and normal-crossings discriminant \cite{N}.

\section{Quasianalytic classes}\label{sec:quasian}

We follow the axiomatic framework of \cite{BMsel}. Consider a class of functions $\cQ$
given by the association, to every 
open subset $U\subset \IR^n$, of a subalgebra $\cQ(U)$ of $\cC^\infty (U)$ containing
the polynomial functions and stable under composition with a $\cQ$-mapping (i.e., a mapping
whose components belong to $\cQ$). We say that $\cQ$ is \emph{quasianalytic}
if it satisfies the following three axioms:

\begin{enumerate}
\item\label{axiom_division} \emph{Closure under division by a coordinate.} If $f \in \cQ(U)$ and
$$
f(x_1,\dots, x_{i-1}, a, x_{i+1},\ldots, x_n) = 0,
$$
where $a \in \IR$,  then $f(x) = (x_i - a)h(x),$ where $h \in \cQ(U)$.

\smallskip
\item\label{axiom_inverse_mapping} \emph{Closure under inverse.} Let $\varphi : U \to V$
denote a $\cQ$-mapping between open subsets $U$, $V$ of $\IR^n$.
Let $a \in  U$ and suppose that the Jacobian matrix
$$
\frac{\partial \varphi}{\partial x} (a) := \frac{\partial
(\varphi_1,\ldots, \varphi_n) }{\partial (x_1,\ldots, x_n)}(a)
$$
is invertible. Then there are neighbourhoods $U'$ of $a$ and $V'$ of 
$b := \varphi(a)$, and a $\cQ$-mapping  $\psi: V' \to U'$ such that
$\psi(b) = a$ and $\psi\circ \varphi$  is the identity mapping of
$U '$.

\smallskip
\item\label{axiom_quasianalytic} \emph{Quasianalyticity.} If $f \in \cQ(U)$ has Taylor expansion zero
at $a \in U$, then $f$ is identically zero near $a$.
\end{enumerate}

\begin{remark}\label{rem:axioms} Since the class $\cQ$ is closed under composition, axiom \eqref{axiom_division}
implies that, if $f \in \cQ(U)$, then all partial derivatives of $f$ belong to $\cQ(U)$. Axiom \eqref{axiom_inverse_mapping}
implies that the implicit function theorem holds for functions of class $\cQ$. 
\end{remark} 

Throughout the paper, we work with a fixed quasianalytic class $\cQ$.
The elements of $\cQ$ will be called \emph{quasianalytic functions}. 
A category of manifolds and
mappings of class $\cQ$ can be defined in a standard way. The category of $\cQ$-manifolds is closed under
blowing up with centre a $\cQ$-submanifold \cite{BMsel}.

\section{Resolution of singularities}\label{sec:res}

Resolution of singularities of a finitely generated ideal in a quasianalytic class $\cQ$ was proved in
\cite{BMinv}, \cite{BMsel}. The proof does not in fact require
that the ideal be finite.

\begin{theorem}[resolution of singularities in a quasianalytic class]\label{thm:res}
Let $U$ be an open subset of $\IR^n$ (or a $\cQ$-manifold), and let $\cI$ denote a sheaf of ideals of 
quasianalytic functions on $U$. Assume that each point of $U$ admits a neighbourhood $U'$
such that $\cI|_{U'}$ is generated by a family of sections of $\cI$ over $U'$ (not necessarily finite).
Let $K$ be a compact subset of $U$. Then there is a neighbourhood $W$ of $K$ in $U$, and a mapping
$\pi: W' \to W$ given by a composite of finitely many admissible blowings-up
(so $W'$ is a $\cQ$-manifold), such that the pull-back
$\pi^*\cI$  is a simple normal-crossings divisor.
\end{theorem}

Let $\cQ_U$ denote the sheaf of germs of functions of class $\cQ$ on $U$. When an ideal sheaf $\cI \subset \cQ_U$ satisfies the assumption in Theorem \ref{thm:res}, one says that $\cQ_U / \cI$
is \emph{quasicoherent}.

\begin{remark}[proof of Theorem \ref{thm:res}] 
Every step of the proof of desingularization of an ideal (as presented in \cite{BMfunct},
for example) goes over
to the case that $\cQ_U / \cI$ is quasicoherent, with essentially no change. 
We recall that the desingularization
algorithm is based on resolution of singularities of a \emph{marked ideal} given locally by $\cI$
together with its maximum order $d$. The \emph{order} $\ord_a \cI$ of $\cI$ at $a \in U$ is the
minimum order of elements of $\cI_a$. Clearly, $\ord\,\cI$ is upper-semicontinuous in the
quasianalytic Zariski topology. Resolution of singularities of the marked ideal $(\cI, d)$ involves
recursively constructing and resolving associated marked ideals on local \emph{maximal contact}
subspaces of increasing codimension. The only place where the proof is not obviously the
same as in the case of a finitely-generated ideal is in \cite[Step\,II, p.\,628]{BMfunct}
--- factorization of $\cI$ as the product $\cM(\cI)\cdot\cR(\cI)$ of its \emph{monomial part} $\cM(\cI)$
and \emph{residual} (or \emph{nonmonomial}) \emph{part} $\cR(\cI)$. We need to show that
if $\cQ_U/\cI$ is quasicoherent, then $\cQ_U / \cR(\cI)$ is quasicoherent; this is a simple
consequence of the division axiom \eqref{axiom_division} in Section \ref{sec:quasian}.
\end{remark}

\section{Preliminary lemmas}\label{sec:prelim}

The lemmas of this section are needed for our proof of Theorem \ref{thm:main}. We use $\IN$
to denote the nonnegative integers.

\begin{lemma}\label{lem:contin}
Let $f:U \to \IR$ be an arc-quasianalytic function, where $U\subset \IR^n$ is a connected open set. Assume that $G(x,f(x))\equiv 0$, where $G$ is a quasianalytic function that is not identically zero. 
Then $f$ is continuous.
\end{lemma}

\begin{proof}
Let $\Ga$ denote the graph of $f$. We have to show that, for any $x_0 \in U$, 
$(x_0, f(x_0))$ is the unique limit point of $\Ga$ (including $\infty$) over $x_0$. 
Given a finite limit point $p$ of $\Ga$ over $x_0$, it follows from resolution of 
singularities of $G$ that there is a
quasianalytic arc $\g(t)$, $t \in (-\ep, \ep)$, such that $\g(0) = p$ and $\g(t)$ is
a smooth point of $\Ga$, for any $t \neq 0$. (See, for example, \cite[Thm.\,6.2]{BMsel}.)
Since $f$ is arc-quasianalytic, $f$ is continuous on every quasianalytic arc (in particular, on
the projection of $\g$ to $U$), so that $p = (x_0, f(x_0))$. On the other hand, suppose
that $\Ga$ has an infinite limit point over $x_0$. Consider sequences $\{a_n\}$ and $\{b_n\}$ 
in $\Ga$ tending to $(x_0, f(x_0))$ and $\infty$. Then the straight lines joining 
the projections of $a_n$, $b_n$ lift to quasianalytic curves in $\Ga$, providing additional 
finite limit points over $x_0$ as $n$ tends to $\infty$ (a contradiction). 
\end{proof}

\begin{lemma}\label{lem:ineq}
Let $U$ be a neighbourhood of the origin in $\IR^n$, and let $c:U\to \IR$ be a quasianalytic function. Assume that $|c(x)|\leq C |x^\alpha|$, for $x$ near zero in the nonnegative quadrant 
$\{x_1\geq 0, \ldots, x_n \geq 0\}$, where $\alpha \in \IN^n$ and $C$ is a positive constant.  Then the function $x^{-\alpha} c(x)$ is quasianalytic .
\end{lemma}

\begin{proof}
Take $i$ such that $\alpha_i\neq 0$. It is enough to show that $c$ is divisible by $x_i$, because we 
can iterate the argument $\alpha_i$ times, for each $i$. Since  $|c(x)|\leq C |x^\alpha|$ in the
nonnegative quadrant, and $\alpha_i\neq 0$, the function $c$ is identically zero on the hyperplane $x_i=0$. By axiom \eqref{axiom_division}, this implies that $c$ is divisible by $x_i$.
\end{proof}

Given $x \in \IR$ and $k \in \IN\setminus \{0, 1\}$, we write $x^{1/k}$ for the positive $k^{\text{th}}$ root of $|x|$. (This
somewhat unusual convention is convenient for Lemma \ref{lem:cont} following, and will intervene again in 
Section \ref{sec:main} only in the same way, for the purpose of applying Lemma \ref{lem:cont}.) Let 
$\IQ_+$ denote the set of nonnegative rational numbers. Given a neighbourhood $U$ of the origin and an $n$-tuple $p \in \{0,1\}^n$ we set
$$
U^p :=\{ x \in U: (-1)^{p_1} x_1\geq 0,\ldots, (-1)^{p_n}x_n\geq 0\}. 
$$

\begin{lemma}\label{lem:cont}
Let $f:U\to \IR$ be a continuous arc-quasianalytic function in a neighbourhood of 
$0 \in \IR^n$, and let $g(x):=x^{-\alpha} f(x)$, where $\alpha \in \IQ^n_+$. Assume  that $|g|\leq M$ 
near $0$, for some $M>0$, and that, for any $p \in \{0,1\}^n$, there is a quasianalytic function 
$G = G_p: U^p \times (-2M,2M)\to \IR$ such that $G$ is not identically zero and
$G(x,g(x))\equiv 0,$ for any $x\in U^p$ near $0$ with $x^\alpha \neq 0$. Then  $|g|$ extends to a continuous function in a neighbourhood of $0$.

Moreover, if $\alpha \in \IN^n$, then $g$ extends to a continuous function in a neighbourhood of $0$.
\end{lemma}

\begin{proof} It is enough to show that $|g|$ or $g$  extends continuously to $0$ since the argument will
apply to every $x$ in a neighbourhood of $0$. Consider a connected component $C$ of $\IR^n \setminus \{x^\alpha=0\}$. We first show that $g_{|C}$ extends continuously to $0$, by 
contradiction.  Suppose that $g|_{C}$ has two distinct asymptotic values $a$ and $b$ at $0$. Since 
$C$ is locally connected at $0$ and $f$ is continuous, all points of the interval $(a,b)$ are also asymptotic values of $g|_{C}$. Therefore, the function $y\mapsto G(0,y)$ is identically zero on 
$(a,b)$. Since $G(0,y)$ is quasianalytic, $G(0,y)$ is identically zero,  
by axiom (\ref{axiom_quasianalytic}); a contradiction.

For each connected component $C$ of $\IR^n \setminus \{x^\alpha=0\}$, there thus exists precisely
one asymptotic value $\lambda_C$ of $g$ at $0$. We have to prove that the $|\lambda_C|$ coincide. Indeed, by axiom (\ref{axiom_division}), since $f$ is arc-quasianalytic and $\alpha \in \IQ^n$,  
there is a positive integer $u$ such that the composite of $g^u$ with a quasianalytic arc 
$\gamma$ is a quasianalytic function, 
and therefore continuous at $0$. This means that $|g|$ extends continuously along every segment passing from one connected component of $\IR^n \setminus \{x^\alpha=0\}$ to another (since $|g|=(g^u)^{1/u}$). Consequently,  the respective asymptotic values of $|g|$ must match.

Assume, moreover, that $\alpha \in \IN^n$. Since $f$ is arc-quasianalytic, $g$ extends to a 
quasianalytic function on every segment. For the same reason as above, therefore, $g$ is continuous.
\end{proof}

Lemma \ref{lem:machine} following is a reformulation of the basic idea of \cite{BMarc}.

\begin{definition}\label{def:reg}
Let $U$ denote an open subset of $\IR^n \times \IR$.
A function $G(x,y)\in \cQ(U)$ is $y$\emph{-regular of order} $k$ at $(x_0,y_0)\in U$ if  
$(\partial^k G / \partial y^k) (x_0,y_0)\neq 0$, and $(\partial^j G / \partial y^j) (x_0,y_0)=0 $ for $j<k$.
\end{definition}

In particular, $G$ is $y$-regular of order $0$ at $(x_0,y_0)$ if $G(x_0,y_0)\neq 0$. 

\begin{lemma}\label{lem:machine}
Let $U$ be an open neighbourhood of $(0,0)$ in $\IR^n \times \IR$, and
let $G: U \to \IR $ denote a quasianalytic function which is $y$-regular of order $d>1$
at $(0,0)$. Set
$c_i(x):=(\partial^i G / \partial y^i)(x,0)$, $0 \leq i \leq d$. Assume that
$c_{d-1}\equiv 0$ and that $c_i(x)=x^{\alpha (d-i)} c^*_i(x)$ (i.e., $c_i(x)$ is divisible by $x^{\alpha (d-i)}$),
for $0 \leq i< d-1$, where $\alpha \in \IN^n$ and where $c^*_k(0) \neq 0$ for some  $k <d-1$. Then 
$$
H(x,y):=x^{-\alpha d}G(x,x^{\alpha }y)
$$ 
is a quasianalytic function which is $y$-regular of order at most $k$ at $(0,0)$.
Moreover, if $\al \neq 0$, then $H(x,y)$ is quasianalytic and
$y$-regular of order at most $d-1$ at $(0,y_0)$, for any $y_0 \in \IR$
(not necessarily near $0$).
\end{lemma}

\begin{proof}
Given $x$  in a neighbourhood of $0$, Taylor's formula for the function $y \mapsto G(x,y)$ at $0$ gives
$$
G(x,x^\alpha y)= \sum_{i=0}^{d-1}c_i(x) x^{\alpha i}\frac{ y^{i}}{i!}+c_d(x)x^{\alpha d}y^{d}\rho(x,x^{\alpha }  y),
$$
where $\rho$ is a $C^\infty$ function such that $ \rho (0,0)\neq 0$.
Clearly, if $\al \neq 0$, then the latter equation holds for $y$ in a neighbourhood of any fixed $y_0$ in $ \IR$, 
and $x$ close enough to zero. 

Since each $c_i (x)=x^{\alpha(d-i)} c^*_i (x)$,  we get
\begin{equation}\label{eq:G}
G(x,x^\alpha y)= \sum_{i=0}^{d-1}x^{\alpha d}c_i^*(x)  \frac{ y^{i}}{i!}+x^{\alpha d}c_d(x)y^{d}\rho(x,x^{\alpha}  y).
\end{equation}
Thus $G(x,x^\alpha y)$ is divisible by $x^{\alpha d}$, and
\begin{equation}\label{eq:H}
H(x,y)= \sum_{i=0}^{d-1}c_i^*(x) \frac{ y^{i}}{i!}+c_d(x)y^d \rho(x,x^{\alpha }  y)
\end{equation}
is  a quasianalytic function (by axiom 
\eqref{axiom_division}). 

Let $k$ denote the smallest integer such that $c_k^*(0)\neq 0$. We will show that $H$ is $y$-regular 
of order $k$ at $(0,0)$, and, if $\al \neq 0$, then $H$ is $y$-regular of order at most $d-1$ at 
any $(0,y_0)$.
>From \eqref{eq:H}, we get 
\begin{equation}\label{eq:taylorH}
\frac{\partial^{k}H}{\partial y^{k}}(x,y)=
 c^*_{k}(x)+\sum_{i=k+1}^{d-1}c^*_i(x) \frac{ y^{i-k}}{(i-k)!}+c_d(x)  y^{d-k}\sigma(x,x^{\alpha }  y),
\end{equation}
where $\sigma$ is a $\cC^\infty$ function. Thus,
\begin{equation}\label{eq:H0}
\frac{\partial^{k}H}{\partial y^{k}}(0,0)=c_k^*(0)\neq 0;
\end{equation}
i.e., $H$ is $y$-regular of order $k$ at $(0,0)$.

Now suppose that $\al \neq 0$ and consider some nonzero $y_0 \in \IR$. Since  $c_{d-1}\equiv 0$, \eqref{eq:H} also gives
\begin{equation}\label{eq:taylorH2}
\frac{\partial^{d-1}H}{\partial y^{d-1}}(x,y)=c_d(x)y \tau(x,x^{\alpha }  y),
\end{equation}
for $(x,y)$ near enough to $(0,y_0)$, where $\tau$ is a $\cC^\infty$ function such that $\tau(0,0)\neq 0$.  
Since $c_d(0)\neq 0$, the right-hand side of \eqref{eq:taylorH2} is clearly nonzero at $(0,y_0)\neq (0,0)$, as required.
\end{proof}

The following lemma is a variation of  Lemma \ref{lem:machine} that will
provide an alternative argument when the hypotheses of Lemma \ref{lem:machine} are not satisfied. 
We drop the assumption
that $c_{d-1} \equiv 0$, but can conclude that $H$ is $y$-regular of order $< d$ only near $(0,0)$.

\begin{lemma}\label{lem:machine2}
Let $G(x,y)$ be a quasianalytic function which is $y$-regular of order $d>0$ at $(0,0)$. Set 
$c_i(x):=(\partial^i G / \partial y^i)(x,0)$, $0 \leq i \leq d$.
Assume that 
$c_i(x)=x^{\alpha (d-i)} c^*_i(x)$,
for $0 \leq i< d$, where $\alpha \in \IN^n$ and where $c^*_k(0) \neq 0$ for some  $k <d$.
Then $H(x,y):=x^{-\alpha d}G(x,x^{\alpha }y)$ is a quasianalytic function in
a neighbourhood of $(0,0)$, and $H$ is $y$-regular of
order at most $k$ at $(0,0)$.
\end{lemma}

\begin{proof}
Again by \eqref{eq:G}, $H$ is a quasianalytic function. Let $k$ be  the smallest integer $<d$ such that
$c^*_k(0) \neq 0$. It follows from \eqref{eq:taylorH} and \eqref{eq:H0} that $H$ is $y$-regular of order
$k$ at $(0,0)$.
\end{proof}

\section{Proof of the main theorem}\label{sec:main}

\begin{proof}[Proof of Theorem \ref{thm:main}]
The ``if'' direction is clear. We will prove ``only if''.
The problem is local, so we work in a neighbourhood of fixed $x_0 \in U$.

By assumption, there is a quasianalytic function $G: U \times (a,b) \to \IR$, not identically zero, and 
vanishing on the graph of $f$. We will first check that we can assume, without loss of generality, that  
$G(x_0,\cdot)$ is not identically zero.

Let $\omega_i:=(\partial^i G / \partial y^i)(x,f(x_0))$, $i \in \IN$. Apply Theorem \ref{thm:res} to the ideal  
sheaf generated by  the 
functions $\omega_i$, $i \in \IN$. This provides a composite of admissible blowings-up
$\pi:U' \to U$ such that, near any point of $z \in\pi^{-1}(x_0)$, if $\cJ_z$ denotes the ideal
generated by the restrictions of the functions $\omega_i \circ \pi$, $i \in \IN$, then, up to a local coordinate system at $z$, 
$\cJ_z$ is generated by a monomial $x^\theta$, $\theta \in  \IN^n$. This implies that every $\omega_i$ is divisible by 
$x^\theta$. We claim it also follows that 
$$
\widetilde{G}(x,y):=x^{-\theta}G(\pi(x),y)
$$
is a quasianalytic function on $U' \times (a,b)$.

To see this, take $i\leq n$ such that $\theta_i\neq 0$, where $\theta=(\theta_1,\dots,\theta_n)$. It is enough to show that 
$G(\pi(x),y)$ is divisible by $x_i$, since we can iterate the argument $\theta_i$ times for each $i$. For any fixed $x$ 
such that  $x_i=0$, the function $\widehat{G}(y):= G(\pi(x),y)$ is quasianalytic and satisfies 
$$
\frac{d^j \widehat{G}}{d y^j} (f(x_0))=\frac{\partial^j G}{\partial y^j} (\pi(x),f(x_0))=0,
$$ 
for every $j \in \IN$. Consequently, by axiom (\ref{axiom_quasianalytic}), $\widehat{G}$ is identically zero. This implies that $G(\pi(x),y)$ is zero on the hyperplane $x_i=0$, so that, by axiom (\ref{axiom_division}), $G(\pi(x),y)$ is divisible by $x_i$, as required.

Since the ideal $\cJ_z$ is generated by $x^{\theta}$, there is an integer $d$ such that $\omega_d \circ \pi$ coincides with 
$x^\theta$ up to multiplication by a local unit.  Thus, for any $z\in\pi^{-1}(x_0)$, there is $d$ such that 
$$
\frac{\partial^d \widetilde{G}}{\partial y^d}(z,f(x_0))\neq 0.
$$

Since we can work both locally and up to a locally finite modification, we will therefore assume that, near $x_0$,
the function $f$ is a root of a quasianalytic function $G(x,y)$ which is $y$-regular of order $d$ at $(x_0,f(x_0))$.

Arguing by induction on $d$, we can assume that the main result is true for every arc-quasianalytic function 
which is a root of a quasianalytic function that is $y$-regular of order $d'<d$.

Since $G$ is $y$-regular of order $d$, the equation
\begin{equation}\label{eq:impl}
\frac{\partial^{d-1} G}{\partial y^{d-1}} (x,y)=0,
\end{equation}
has nonvanishing $y$-derivative at $(x_0,f(x_0))$. Therefore, by the implicit function theorem
(axiom \eqref{axiom_inverse_mapping} and Remark \ref{rem:axioms}) \eqref{eq:impl} implicitly defines
a quasianalytic function $\varphi:U\to \IR$ in a neighborhood of $x_0$ (which we continue to call $U$),
such that $\vp(x_0) = f(x_0)$.  

Set
$$
c_i(x):=\frac{\partial^{i} G}{\partial y^{i}} (x,\varphi(x)),\quad x \in U,
$$
and apply Theorem \ref{thm:res} to the
ideal sheaf $\cI$
generated by the functions $c_i^{d! / (d-i)}$, $i<d-1$. This provides a composite of blowings-up
after which we can assume that $\cI$ is generated by a monomial $x^\alpha$, $\alpha \in \IN^n$; i.e.,
we can assume that $c_i^{d! / (d-i)} = x^\alpha \cdot  c_i^*(x)$, $i=0,\dots,d-2$, where $c_i^*(x)$
is a unit, for some $i$. Without loss of generality, we can also 
assume that $(x_0,f(x_0))=(0,0)$.

For simplicity, write 
$$
g(x):=f(x)-\varphi(x);
$$
then $g(0)=0$.
If $\cI = (0)$, then $g \equiv 0$, which means that $f \equiv \varphi$, and the result is clear.  Otherwise, consider the function
$$
g_1(x):=x^{-\frac{\alpha }{d!}}\cdot g(x),
$$
for  $x=(x_1,\dots,x_n)$  sufficiently close to  $0$ in $\{x: x^\alpha\neq  0\}$. Note that $g_1$ is a root of the function 
$$ 
G_1(x,y):=x^{-\frac{\alpha }{(d-1)!}}\cdot G(x,x^\frac{\alpha}{d!}y+\varphi(x)),
$$
defined for $x^\alpha\neq  0$.

We first check that $g_1$ is bounded. Taking the Taylor expansion
of $y\mapsto G(x,y+\varphi(x))$ and evaluating it at $g(x)$, we see that
\begin{equation}\label{eq:taylorG1}
\sum_{i=0}^{d-1}\left|\left(c_i^*\right)^{\frac{d-i}{d!}}(x)\cdot
x^{\frac{\alpha(d-i)}{d!}}\cdot \frac{ g(x)^{i}}{i!}\right| \geq
\left|\rho(x,g(x))g(x)^d\right|,
\end{equation}
where $\rho(x,t)$ is a $\cC^\infty$ function that does not vanish at $(0,0)$.
Assume that $g_1$ is not bounded. Then there is a sequence $x_\nu$ tending to zero such that
$$
|g(x_\nu)| \gg |x^\frac{\alpha}{d!}_\nu|,
$$
and therefore 
$$
|g(x_\nu)|^d \gg
|x_\nu ^{\frac{\alpha(d-i)}{d!}}|\cdot |g(x_{\nu})|^{i},
$$
for any $i < d$.  This contradicts \eqref{eq:taylorG1};
therefore $g_1$ is bounded.

By  Lemma \ref{lem:machine}, the  function 
$$
G_1(x^{d!},y)=x^{-\alpha d}\cdot G(x^{d!},x^{\alpha} y+\varphi(x^{d!}))
$$ 
is  quasianalytic and $y$-regular of
order at most $(d-1)$ at any point of the $y$-axis. This shows that $G_1$ extends to a continuous function  on the quadrant 
$\{(x,y):x_1\ge 0,\dots,x_n\ge 0\} $. More generally, given any $p \in \{0,1\}^n$, by the same argument (applying 
Lemma \ref{lem:machine} to the function $G_1((-1)^{p_1}x^{d!}_1,\ldots,\allowbreak (-1)^{p_n} x_n^{d!},y)$), we see that $G_1$ is continuous on  $U^p\times [-M,M]$,  if $U$ is a sufficiently small neighborhood of the origin in $\IR^n$ and $M >0$. Therefore, by Lemma \ref{lem:cont}, $|g_1|$ extends continuously to the origin.

\medskip
The problem is that the function $g_1$ might not be  arc-quasianalytic since $\alpha / d!$ \emph{a priori} 
is not an element of $\IN^n$. If $g_1$ does not tend to zero at the origin,
we will see that $\alpha / d!$ is necessarily an integer, and then  $G_1$ is quasianalytic. (This is Case I below.) The most
difficult case is when $g_1$ tends to zero at the origin (Case II). In this case, we
will see that we can iterate the preceding argument finitely many times, and find $G_2,\dots, G_k$ and
$g_2,\dots,g_k$, until $\lim_{x\to 0} |g_k(x)|\neq 0$. The order  of $y$-regularity of the
$G_j$ will be strictly decreasing, forcing the process
to end after finitely many steps. The construction of the successive $G_j$ will be done by a method very similar to that
above; nevertheless, for technical reasons, we have to  replace
Lemma \ref{lem:machine} with Lemma \ref{lem:machine2}.

We thus distinguish two cases.

\medskip\noindent
\emph{Case I.\,} $\lim_{x \to 0} |g_1(x)|\neq 0$.
By construction, 
\begin{equation}\label{eq:g_k}
g_1(x)=x^{-\gamma}(f(x)-\varphi (x)),
\end{equation}
for some $\gamma \in \IQ_+^n$. We claim that $\gamma \in \IN^n$.

To see this,
consider the curve $\beta(t)=(t,\sigma,\dots,\sigma)$ where $\sigma$ is a sufficiently small positive real
number. Since $\varphi$  is quasianalytic and $f$ is arc-quasianalytic,  $(f-\varphi)(\beta(t))$ is a quasianalytic function.  
Set  $\gamma=(\gamma_1,\dots, \gamma_n)$. Since
$|g_1|$ is continuous and does not vanish at $0$, then, by \eqref{eq:g_k},
$\gamma_1$ equals the  exponent of the first term of the Taylor expansion of the arc $(f-\varphi)(\beta(t))$. 
Thus $\gamma_1$ is an integer. Of course, we can repeat the argument for every $\gamma_j$.

Now, by Lemma \ref{lem:ineq}, $G_1$ is quasianalytic (since $G_1$ is bounded on the nonnegative quadrant). 
Moreover, by Lemma \ref{lem:cont}, $g_1$ extends continuously to an arc-quasianalytic function in a neighbourhood 
of the origin. Therefore, $g_1$ is an arc-quasianalytic root of
a quasianalytic function which is $y$-regular of order $i<d$ in a neighbourhood of $0$.  By induction on the order 
of regularity, $g_1$ becomes quasianalytic after a local modification. By \eqref{eq:g_k}, therefore so does $f$, as required.

\medskip\noindent
\emph{Case II.\,}  $\lim_{x \to 0} g_1(x)=0$. We can assume that $g$ is not identically zero since
otherwise the result is clear. Case II ultimately reduces to Case I after iterating an argument similar to that used to construct $G_1$ and $g_1$.
Given $u \in \IN$ and $x=(x_1,\ldots,x_n)$, we set $x^u:=(x_1^u,\ldots,x_n^u)$. We first prove the following.

\begin{claim}\label{claim}
Let $G:W \times (a,b) \to \IR$ be a quasianalytic function (not identically zero), where $W$ is a neighbourhood of $0$
in $\IR^n$, and let $g:W \to (a,b)$ be an arc-quasianalytic function such that $G(x,g(x))\equiv 0$. Let $\pi_1:V \to W$ be 
a local modification, and let $A:V \to \IR$ be a  unit (i.e., a nowhere vanishing quasianalytic function).

Given $\alpha_1, \beta_1 \in \IQ^n_+$, set 
$$
G_1(x,y):=x^{-\beta_1} G(\pi_1(x), A(x)\cdot x^{\alpha_1} y),
$$ 
and
$$
g_1(x):=\frac{x^{-{\alpha_1}}}{A(x)}\cdot  g(\pi_1(x)),
$$
(so that $g_1(x)$  is a root of $G_1(x,y)$).

Assume that $\alpha_1$ and $\beta_1$ are such that $G_1(x^u,y)$ is quasianalytic and $y$-regular of order at
most $i>0$ at all $(x,g_1(x))$, $x \in V$, for some positive integer $u$. Assume also that $\lim_{x \to 0} g_1(x)=0$.

Then there exists  a locally finite  covering by modifications such that, for each member
 $\pi_2:U \to W$ of this covering, where $U$ is a neighbourhood of the origin and $\pi_2(0)=0$, 
 there are  $\alpha_2, \beta_2 \in \IQ^n_+$ and a unit $B:U \to \IR$ such that:
\begin{enumerate}
\item If
\begin{equation}\label{eq:G_2}
G_2(x,y):=x^{-\beta_2}\cdot G(\pi_2(x), B(x)\cdot x^{ \alpha_2}y), 
\end{equation}
then, for a suitable positive integer $u'$, 
$G_2(x^{u'},y)$ is a quasianalytic function which is $y$-regular of order $i'<i$ at any $(a,0)$, 
$a  \in  \pi_2^{-1}(0)$.
\item The function
$$g_2(x):=\frac{x^{ -\alpha_2}}{B(x)}\cdot g(\pi_2(x))
$$
(which is a root of $G_2(x,y)$) is bounded on $U\setminus \{x^{\alpha_2}=0\}$.
\end{enumerate}
\end{claim}

\begin{proof}[Proof of Claim \ref{claim}]
Note that the hypotheses of the claim are preserved by a small translation $(x,y) \mapsto (x+x_0,y)$, 
$x_0 \in \pi_1^{-1}(0)$. Therefore, we can focus nearby a point of $V$ and assume it is the origin 
(and we can also assume $\lim_{x\to 0}g_1(x)=0$, since otherwise we are in Case I).

We will show that each of the functions
$$
\lambda_j(x) :=\frac{\partial ^j G_1}{\partial y^{j}}(x,0)^u, \quad j=0,\ldots, i-1,
$$
is quasianalytic, for $u$ as in the hypotheses of the claim. Indeed, $\la_j$ coincides (up to multiplication by a unit)
with
$$
x^{u(j\alpha_1-\beta_1)}\cdot \frac{\partial^j G}{\partial y^j}(\pi_1(x), 0)^u.
$$ 
Since $(\partial ^j G_1 / \partial y^{j})(x^u,0)$  is quasianalytic, the function 
$$
x^{u(j\alpha_1 -\beta_1)}\cdot \frac{\partial ^j G}{\partial y^{j}}(\pi_1(x^u),0)
$$ 
is  bounded near $0$. Consequently, in some neighbourhood of $0$, 
$$
\left|\frac{\partial ^j G}{\partial
y^{j}}(\pi_1(x^u),0)\right|\leq C \left|x^{u(\beta_1-j\alpha_1)}\right|,  
$$
where $C$ is a positive constant.
Substituting $x^{1/u}$ for $x$ in this inequality and raising both sides to the power $u$, we get 
$$
\left|\frac{\partial ^j G}{\partial
y^{j}}(\pi_1(x),0)^u\right|\leq C^u \left|x^{u(\beta_1-j\alpha_1)}\right|,  
$$
for $x$ near zero in the quadrant $\{x_1\geq 0,\dots,x_n \geq 0\}$.
By Lemma \ref{lem:ineq}, the latter implies that  $\lambda_j$ is quasianalytic.

Now let $\cI$ denote the ideal sheaf generated by the $\lambda_j^{i! / (i-j)}$, $j<i$, and apply 
Theorem \ref{thm:res} to the ideal sheaf $\cJ$ given by the product of $\cI$ and all 
$\lambda_j^{i! / (i-j)}$, $j<i$.
The theorem provides a composite of blowings-up $\pi:U \to V$ such that
$\pi^*\cJ$ is a normal crossings divisor; therefore, $\pi^*\cI$ is a normal crossings divisor
and also each $\lambda_j^{i! / (i-j)}$, $j<i$, becomes a 
monomial times a unit in suitable local coordinates.  We can assume also that the components 
of $\pi$ are monomials times units.  Let $\pi_2:=\pi_1\circ \pi$. 

Since the components of $\pi$ are normal crossings, we have
$$
G_1(\pi(x),y)=A''(x)x^{-\beta_1'} G(\pi_2(x), A'(x)\cdot x^{\alpha_1'} y),
$$
for some $\alpha_1', \beta_1' \in \IQ_+^n$ and some units $A', A''$. It follows from 
Lemma \ref{lem:ineq} that $G_1(\pi(x^u),y)$ is quasianalytic.

Up to multiplication by a unit, 
$$
\frac{\partial^j G_1}{\partial y^j}(\pi(x^u),0) = \lambda_j\circ \pi(x).
$$
Hence the hypotheses of Lemma \ref{lem:machine2} are satisfied by the quasianalytic function $G_1(\pi(x^{u}),y)$.
By Lemma \ref{lem:machine2}, therefore, there are $\alpha_2, \beta_2 \in \IQ_+$, a unit $B:U \to \IR$, 
and a positive integer $u'$, 
such that, if $G_2$ is defined as in \eqref{eq:G_2}, then the function $G_2(x^{u'},y)$  
is quasianalytic and $y$-regular of order at most $i'<i$ at $(0,0)$. This gives (1).

The argument used in \eqref{eq:taylorG1} and immediately following for $g_1$ clearly also 
applies to show that $g_2$ is bounded, giving (2).
This completes the proof of Claim \ref{claim}. 
\end{proof}

We can now finish the proof of Case II and therefore of the theorem. 
Note that $G_1$ and $g_1$ satisfy the assumptions of Claim \ref{claim}. 
We can therefore apply the claim to get $G_2$ and $g_2$,  and show that $|g_2(x)|$ extends continuously to the points 
where $x^\alpha =0$, by the same argument used for $g_1$.

If $\lim_{x \to 0} |g_2(x)|\neq 0$, then we are done, according to Case I. Otherwise, we can apply Claim \ref{claim}
to $G_2$ and $g_2$. Iterating the argument, we get two sequences of functions $G_k$ and $g_k$, $k=1,\ldots,l$, such that, for every $k$, $\lim_{x \to 0} g_k (x)=0$ and $G_k(x^{u},y)$ is $y$-regular of order $i_k$ at $(0,0)$, for suitable $u$, and
$i_2,\dots , i_l$ is a strictly decreasing sequence of positive integers. Of course, such a sequence cannot continue indefinitely, 
so that $\lim_{x \to 0} |g_l (x)| \neq 0$, for some $l$. (If $G_l$ is $y$-regular of order $0$ at $(0,0)$, then $G_l(0,0)\neq 0$, so that $\lim_{x \to 0} |g_l (x)|\neq 0$.)
In other words, we are eventually in Case I.
\end{proof}

\bibliographystyle{amsplain}

\end{document}